\documentclass{daj}

\usepackage{amssymb}
\usepackage{amsthm}
\usepackage{amsmath}

\newtheorem{theorem}{Theorem}[section]
\newtheorem{proposition}[theorem]{Proposition}
\newtheorem{lemma}[theorem]{Lemma}
\newtheorem{definition}[theorem]{Definition}
\newtheorem{corollary}[theorem]{Corollary}

\newtheorem{problem}{Problem}[section]

\newcommand{\CC}{\mathbb{C}}

\newcommand{\ZZ}{\mathbb{Z}}
\newcommand{\RR}{\mathbb{R}}
\newcommand{\NN}{\mathbb{N}}

\newcommand{\cala}{\mathcal{A}}

\newcommand{\supp}{{\hbox{supp}\,}}

\newcommand{\Div}{{\mathrm{Div}}}

\dajAUTHORdetails{%
  title = {Functional Tilings and the Coven-Meyerowitz Tiling Conditions}, 
  author = {Gergely Kiss, Itay Londner, M\'at\'e Matolcsi, and G\'abor Somlai},
  plaintextauthor = {Gergely Kiss, Itay Londner, M\'at\'e Matolcsi, G\'abor Somlai},
  keywords = {tiling of cyclic groups, Coven-Meyerowitz conditions},
}   

\dajEDITORdetails{%
   year={2025},
   number={18},
   received={7 November 2024},   
   revised={24 April 2025},    
   published={15 September 2025},  
   doi={10.19086/da.144131},       
}   

\begin{document}

\begin{frontmatter}[classification=text]


\author[gergo]{Gergely Kiss\thanks{Supported by supported by the Hungarian National Foundation for
		Scientific Research, Grants No. K146922, FK 142993 and by the J\'anos Bolyai Research Fellowship of the Hungarian Academy of Sciences.}}
\author[itay]{Itay Londner\thanks{Supported by the Israel Science Foundation (Grant 607/21).}}
\author[mate]{M\'at\'e Matolcsi\thanks{Supported by the Hungarian National Foundation for Scientific Research, Grants No. K132097, K146387.}}
\author[gabor]{G\'abor Somlai\thanks{Supported by the Hungarian National Foundation for Scientific Research, Grant No. 138596 and Starting Grant 150576.}}

\begin{abstract}
Coven and Meyerowitz \cite{CM} formulated two conditions which have since been conjectured to characterize all finite sets that tile the integers by translation. By periodicity, this conjecture is reduced to sets which tile a finite cyclic group $\ZZ_M$. In this paper we consider a natural relaxation of this problem, where we replace sets with nonnegative functions $f,g$, such  that $f(0)=g(0)=1$, $f\ast g=\mathbf{1}_{\ZZ_M}$ is a functional tiling, and $f, g$ satisfy certain further natural properties associated with tilings. We show that the Coven-Meyerowitz tiling conditions do not necessarily hold in such generality. Such examples of functional tilings carry the potential to lead to proper tiling counterexamples to the Coven-Meyerowitz conjecture in the future.
\end{abstract}
\end{frontmatter}


	\section{Introduction}
	\subsection{Tilings of the integers and the Coven-Meyerowitz conjecture}
	Let $A\subset \ZZ$ be a finite set. One says that $A$ \emph{tiles the integers by translations} if a disjoint union of translated copies of $A$ covers $\ZZ$. Alternatively, if there exists a set of translations $T\subset \ZZ$  such that any integer $n$ can be uniquely expressed as the sum $a+t=n$, with $a\in A$ and $t\in T$. It is well-known \cite{New} that in this case the set $T$ must be periodic, i.e. there exist $M\in \NN$ and a finite set $B\subset \ZZ$ such that  $T=B\oplus M\ZZ$. This implies that $|A||B|=M$ and $A\oplus B$ is a tiling of the cyclic group $\ZZ_M$, with addition$\mod M$.
	
	\medskip
	
	The prime factorization of $M$ will be denoted by $M=\prod_{i=1}^K p_i^{n_i}$.
	
	\medskip
	
	Let $f:\ZZ_M\rightarrow \CC$ be an arbitrary function. Identifying $\ZZ_M$ with $\{0,\ldots, M-1\}$, we define the \emph{mask polynomial} of $f$ as
	\begin{equation}\label{maskpoly}
		F(X)=\sum_{z\in\ZZ_M} f(z)X^z.
	\end{equation}
	
	The mask polynomial is closely
	related to the Fourier transform of $f$, defined as
	$$
	\hat{f}(\xi)=\sum_{z\in\ZZ_M} f(z)e^{-2\pi i z\xi/M} \ \ (\xi=0, 1, \dots, M-1).
	$$
	Namely, with the notation $\eta=e^{-2\pi i/M}$, we have $F(\eta^\xi)=\hat{f}(\xi)$.
	
	If $f$ is nonnegative, the total weight of $f$ will be denoted by

    \begin{equation}\label{fabs}
    |F|:=F(1)=\sum_{z\in\ZZ_M} f(z)=\hat{f}(0).
    \end{equation}

	Given $m|M$, we define the set
	$$
	R_m:=\{z\in\ZZ_M:(z,M)=m\}.
	$$
	Note, for future reference, that $R_M=\{0\}$. We refer to each of the sets $R_m$ as a {\it step class}.
	
	\begin{definition}\label{stepf}
		We say that $f$ is a {\em step function} if $f$ is constant on each step class $R_m$. We will denote the set of step functions by $\mathcal{A}(M)$.
	\end{definition}
	
	The set $\mathcal{A}(M)$ is a vector space with the additional property that the Fourier transform maps it to itself (see Lemma \ref{fl} below).
	
	\medskip
	
	Given a set $A\subset \ZZ_M$, its mask polynomial, defined in \eqref{maskpoly} with $f=\mathbf{1}_A$, will be denoted by $A(X)$. As $A(X)$ is a polynomial with integer coefficients, for any $m|M$ the Fourier transform $\hat{1}_A$ is either constant zero on the step class $R_m$, or it has no zeros in $R_m$ at all. As such, the support of $\hat{1}_A$ is always a union of step classes. Also, the total weight of $\mathbf{1}_A$ coincides with the cardinality of the set $A$, which will be denoted $|A|$.
	
	Using the language of mask polynomials one may reformulate the tiling property in $\ZZ_M$ in terms of cyclotomic polynomials. Due to translation invariance we may assume  $0\in A\cap B$. The tiling condition $A\oplus B=\ZZ_M$ is equivalent to
	\begin{equation}\label{polytile}
		A(X)B(X)=1+X+\ldots +X^{M-1}\mod(X^M-1).
	\end{equation}
	Let $\Phi_m$ be the cyclotomic polynomial of order $m$, that is the unique, monic, irreducible polynomial whose roots are the primitive $m$-th roots of unity. The factorization \eqref{polytile} is in turn equivalent to
	\begin{equation}\label{cycfact}
		|A||B|=M\text{ and for all } m|M,\, m\neq 1\,  \Phi_m(X) \text{ divides at least one of } A(X), B(X),
	\end{equation}
	or, using the language of Fourier transforms, it is further equivalent to
	\begin{equation}\label{fou}
		|A||B|=M\text{ and for all } 0\ne \xi\in \ZZ_M, \
		\hat{\mathbf{1}}_A(\xi)=0 \ \text{or} \  \hat{\mathbf{1}}_B(\xi)=0.
	\end{equation}
	
	Note that for any $m|M$ we have

    \begin{equation}\label{cycfou}
    \Phi_m(X)|A(X) \ {\text{if and only if}} \ A(\eta^\xi)=0 \ {\text{for all}} \  \xi\in R_{M/m},
    \end{equation}
which can also be stated as
\begin{equation}\label{cycfou2}
\Phi_m(X)|A(X) \ {\text{if and only if}} \
\hat{\mathbf{1}}_A(\xi)=0  \ {\text{for all}} \ \xi\in R_{M/m}.
 \end{equation}

	\medskip
	
	Coven and Meyerowitz proved the following.
	\begin{theorem}[\cite{CM}]\label{CMthm}
		Let $A$ be a finite set of integers and
		$$
		S_A:=\{p^\alpha\,:\, p^\alpha \text{ is a prime power and } \Phi_{p^\alpha}(X)\mid A(X)\}.
		$$
		Consider the following two conditions

		{\it (T1) $A(1)=\prod_{s\in S_A}\Phi_s(1)$,}
		
		\smallskip
		{\it (T2) if $s_1,\dots,s_k\in S_A$ are powers of different
			primes, then $\Phi_{s_1\dots s_k}(X)\mid A(X)$.}
		\smallskip
		
		\noindent
		Then
		\begin{itemize}
			\item  If $A$ satisfies both (T1) and (T2) then $A$ tiles $\ZZ$.
			\item  If $A$ tiles $\ZZ$  then it must satisfy (T1).
			\item If $A$ tiles $\ZZ$, and $|A|$ has at most two prime factors, then it satisfies (T2).
		\end{itemize}
	\end{theorem}

	The question whether conditions (T1) and (T2) characterize all sets which tile the integers, without any restriction on the prime factorization of the cardinality of a tile, has been referred to in the literature as the \emph{Coven-Meyerowitz conjecture} (the first instance it was referred to as an explicit conjecture is \cite{KL}). This conjecture is considered to be the main open question in the topic of integer tilings. Until recent years, the conjecture has been verified only in some special cases \cite{dutkay-kraus,M, Tao-blog, shi, KL}. More recently there has been noticeable progress. In \cite{LaLo2, LaLo4} it has been confirmed for all tilings of period $M=(p_1p_2p_3)^2$, where $p_1, p_2, p_3$ are distinct primes. Subsequently, in \cite{LaLo3} it has been confirmed for tilings of period $M$ having arbitrarily many prime factors with some specified restrictions on their sizes and powers.
	\medskip
	
	Conditions (T1) and (T2) can also be naturally formulated for arbitrary nonnegative functions $f: \ZZ_M\to \RR$. For the mask polynomial $F(X)$ of $f$, for any $m|M$ we say that $\Phi_m(X)|F(X)$ if $F(\eta^\xi)=0$ for all $\xi\in R_{M/m}$. Then, letting $S_F$ denote the set of prime-powers $p^\alpha|M$ such that $\Phi_{p^\alpha}(X)|F(X)$, we say that $F$ satisfies condition (T1) if $F(1)=\prod_{s\in S_F} \Phi_s(1)$, and $F$ satisfies condition (T2) if $s_1, \dots , s_k\in S_F$ being powers of different primes implies that $\Phi_{s_1\dots s_k}(X)|F(X)$.
	
	\medskip
	
	Going back to condition \eqref{cycfact}, given a collection $H=\{m_1, m_2, \dots, m_s\}$ of divisors of $M$,
	the main difficulty is to decide whether there exists a set $A$ such that the cyclotomic divisors of $A(X)$ are exactly $\Phi_{m_j}$ ($1\le j\le s$), and $A$ satisfies condition (T1). It is not hard to see that this question is equivalent to solving an integer programming (IP) problem  (we will discuss this fact in Section \ref{sec2} below). It is therefore natural to consider a linear programming relaxation of the tiling criteria, to which we turn next.
	
	\subsection{Weak-tiling and pd-tiling}
	
	The notion of {\it weak tiling} was  introduced in \cite{LM} as a key concept in proving Fuglede's conjecture for convex bodies. In the context of finite cyclic groups, it is a relaxation of proper tilings in the following way: a set $A$ is said to tile $\ZZ_M$ weakly if there exists a nonnegative function $f: \ZZ_M\to \RR_+$ such that $f(0)=1$ and $\mathbf{1}_A\ast f=\mathbf{1}_{\ZZ_M}$. Clearly, when $A\oplus B=\ZZ_M$ is a tiling, one can choose $f=\mathbf{1}_B$ (where we assumed, without loss of generality, that $0\in B$).
	
	\medskip
	
	Weak tiling is a very natural concept. However, in order to be able to invoke tools from linear programming, the following slightly modified definition was introduced in \cite{KMMS}.
	
	\begin{definition}\label{pdtile}
		A set $A\subset \ZZ_M$ pd-tiles $\ZZ_M$ weakly, if there exists a function $f:\ZZ_M\rightarrow \RR$ satisfying
		
		(i) $f(0)=1$,
		
		(ii) $f\geq 0$,
		
		(iii) $\hat{f}\geq 0$
		
		\smallskip
		\noindent
		and such that
		\begin{equation}\label{weaktile}
			\mathbf{1}_A\ast f=\mathbf{1}_{\ZZ_M}.
		\end{equation}
		To abbreviate the terminology, we will just say that $A$ pd-tiles $\ZZ_M$, dropping the
		term ”weakly”. Any function $f$ with the properties above will be called a pd-tiling complement of $A$.
	\end{definition}
	
	In this terminology, "pd" stands for positive definite. As observed in \cite{KMMS} if $A\oplus B=\ZZ_M$, then $A$ pd-tiles $\ZZ_M$ with the function $\frac{1}{|B|}\mathbf{1}_B\ast \mathbf{1}_{-B}$. In connection with Fuglede's conjecture we remark that if $A\subset \ZZ_M$ is spectral then $A$ also pd-tiles $\ZZ_M$ (cf. \cite{KMMS}).
		
	\medskip
	
	We also note that for any pd-tiling, equation \eqref{weaktile} is equivalent to the analogue of \eqref{fou}:
	
	\begin{equation}\label{pdt}
		|A||F|=M, \ \text{and for all} \ 0\ne \xi\in \ZZ_M \ \text{we have} \ \hat{\mathbf{1}}_A(\xi)=0 \ \text{or} \ \hat{f}(\xi)=0,
	\end{equation}
where $|F|$ denotes the total weight of $f$, as defined in \eqref{fabs}.

	Furthermore, due to the fact that the support of $\hat{\mathbf{1}}_A$ is a union of step classes $R_m$, the function $\hat{f}$ must vanish on those classes, i.e. $F(X)$ must be divisible by the corresponding cyclotomic polynomials $\Phi_{M/m}(X)$.
	
	\medskip
	
	Formally speaking, the concept of pd-tiling is a relaxation of proper tiling. It is therefore natural to ask the following questions.
	
	\medskip
	
	\begin{problem}\label{prob1}
		Is it true that whenever $A$ pd-tiles $\ZZ_M$, then $A$ also tiles $\ZZ_M$ properly?
	\end{problem}

	\begin{problem}\label{prob2}
		Is the Coven-Meyerowitz conjecture true for all pd-tilings, i.e. is it true that for any set $A\subset \ZZ_M$ that pd-tiles $\ZZ_M$, and any a pd-tiling complement $f$ of $A$, both $A$ and $F$ satisfy conditions (T1) and (T2)?
	\end{problem}

	\medskip
	
	While these questions remain open in full generality, in this paper we give a partial answer to Problem \ref{prob1} (see Proposition \ref{primep} below), and solve a functional variant of Problem \ref{prob2}.

	\subsection{Functional pd-tiling}
	
	As in \cite{KMMS}, it is natural to go one step further with the relaxation of the tiling conditions.
	
	\begin{definition}\label{fweakdef}
		Let $f,g:\ZZ_M\rightarrow\RR$ be functions satisfying conditions (i)-(iii) of Definition \ref{pdtile}. We say that the pair $(f,g)$ is a functional pd-tiling of $\ZZ_M$ if
		\begin{equation}\label{fweaktile}
			f\ast g=\mathbf{1}_{\ZZ_M}.
		\end{equation}
	\end{definition}
	
	Clearly, if $A\oplus B=\ZZ_M$ is a tiling, then the functions $h^{(1)}_{A}=\frac{1}{|A|}\mathbf{1}_A\ast \mathbf{1}_{-A}$ and $h^{(1)}_B=\frac{1}{|B|}\mathbf{1}_B\ast \mathbf{1}_{-B}$ form a functional pd-tiling of $\ZZ_M$, and the cyclotomic divisors of $H^{(1)}_A(X)=\sum_{z\in \ZZ_M} h^{(1)}_A(z)X^z$ and $H^{(1)}_B(X)=\sum_{z\in \ZZ_M} h^{(1)}_B(z)X^z$ are the same as those of $A(X)$ and $B(X)$, respectively. As such, condition (T1) holds for both $H^{(1)}_A$ and $H^{(1)}_B$.
	
	\medskip
	
	In the context of the functional pd-tiling $(h^{(1)}_A, h^{(1)}_B)$ above, it is natural to carry out an {\it averaging} procedure which, as we shall see, leads to {\it step functions} $h_A$, $h_B$ while not altering the cyclotomic divisors of the corresponding mask polynomials. This enables us to study the Coven-Meyerowitz conditions in the class of step functions.

	\medskip
	
	Given the functions $h^{(1)}_A, h^{(1)}_B$ above, and an element $r\in \ZZ_M$ such that $(r, M)=1$, we define $h^{(r)}_A(z)=h^{(1)}_A(rz)$ and
	\begin{equation}\label{avg}
		h_A(z)=\frac{1}{\phi(M)}\sum_{r\in R_1} h^{(r)}_A(z),
	\end{equation}
	and similarly for $h_B$. Here $\phi$ denotes Euler's totient function.
	
	\medskip
	
	For any $z, y$ in the same step class $R_d$ there exists an $r\in R_1$, such that $rz\equiv y$  (mod $M$), therefore $h_A(z)=h_A(y)$, and hence
    the functions $h_A$ and $h_B$ are constant over step classes, i.e. they are step functions, $h_A, h_B \in\mathcal{A}(M)$. It is also well known that the Fourier transform of a step function is also a step function, and we cite the formula for the Fourier transform of the indicator function of a single step class $R_m$ as follows:

\begin{lemma}[\cite{LaLo1}, Lemma 4.8]\label{fl} Let $m|M$, then
		$$
\widehat{\mathbf{1}_{R_m}}(\xi)=\sum_{d|(M/m,\xi)}\mu\left (\frac{M}{md}\right )d,
		$$
		where $\mu$ is the M\"{o}bius function.
		\end{lemma}

	The functions $h_A, h_B$ and their Fourier transforms $\hat{h}_A, \hat{h}_B$ are all nonnegative, and their positive values are all bounded away from zero. Indeed, $h_A=\frac{1}{|A|\phi(M)}\sum_{r\in R_1}1_{rA}\ast 1_{-rA}$, and the sum is a nonnegative, integer-valued function, therefore the non-zero values of $h_A$ satisfy

	\begin{equation}\label{eps}
h_A(z)\ge \frac{1}{|A|\phi(M)}\ge \frac{1}{M \phi(M)}=:\delta_M, \ \ \ \ \ \ (\text{if} \ h_A(z)\ne 0),
    \end{equation}
and the same argument holds for $h_B$. Interestingly, the same lower bound also holds for $\hat{h}_A$ and $\hat{h}_B$ for the following reason:
	
	\begin{equation}\label{trace}
		\widehat{h}_A(\xi)=\frac{1}{|A|\phi(M)}\left(\sum_{r\in R_1}1_{rA}\ast 1_{-rA}\right)^{\bigwedge}(\xi)\ge 0,
	\end{equation}
where the expression is nonnegative because the Fourier transform of each term inside the bracket is $|\hat{1}_{rA}|^2(\xi)\ge 0$. Also, the sum in the bracket is integer-valued, and hence its Fourier transform is also integer valued by Lemma \ref{fl} above. Therefore, we conclude that all non-zero values of $\widehat{h}_A$
satisfy
\begin{equation}\label{trh}
\widehat{h}_A(\xi)\ge \frac{1}{|A|\phi(M)}\ge \frac{1}{M\phi(M)},  \ \ \ \ \ \
( \text{if} \   \widehat{h}_A(\xi)\ne 0)
\end{equation}
 and the same argument holds for $\widehat{h}_B$.

	\medskip

The advantage of the averaging procedure above is that the class of step functions is easier to study with tools from linear programming (LP). As such, we will be able to formulate some non-trivial LP conditions for any functional pd-tilings and also for proper tilings of $\ZZ_M$ (see Section \ref{sec2}). The downside of averaging, however, is that when any nonnegative step function $f$ is given, it is a often a nontrivial task to determine whether there exists a set $A\subset \ZZ_M$ such that $f=h_A$.
	
	\medskip

	Consider now the mask polynomials $H_A$ and $H_B$ corresponding to $h_A, h_B$, as defined in equation  \eqref{maskpoly}. We claim that
	the averaging process preserves  cyclotomic divisibility. First, due to \eqref{cycfou}, for every $m|M$ and   every $(r,M)=1$  we have that
	$\Phi_m(X)|H^{(r)}_A(X)$ if and only if $\Phi_m(X)|H^{(1)}_A(X)$. As the values $H^{(r)}_A(\eta^\xi)=\hat{h}_A(\xi)$ are all nonnegative by \eqref{trace}, we conclude that $H_A(\eta^\xi)=0$ if and only if $H^{(r)}_A(\eta^\xi)=0$ for all $r$. That is, $\Phi_m(X)|H_A(X)$ if and only if $\Phi_m(X)|H^{(r)}_A(X)$ for all $r$, which happens if and only if $\Phi_m(X)|H^{(1)}_A(X)$. Therefore, the cyclotomic divisors of $H_A(X)$ and $H_B(X)$ are the same as those of $A(X)$ and $B(X)$.  In particular, condition (T1) of Coven-Meyerowitz is satisfied for both $H_A$ and $H_B$.
	
	\medskip
	
	It is therefore natural to ask the following question.
	
	\begin{problem}\label{prob3}
		Let $(f, g)$ be a functional pd-tiling of $\ZZ_M$ as in Definition \ref{fweakdef}. Assume that both $f$ and $g$ are step functions, and the corresponding mask polynomials $F$ and $G$ both satisfy condition (T1). Do $F$ and $G$ necessarily satisfy condition (T2)?
	\end{problem}
	
	\medskip
	
	A positive answer to this question would mean that the Coven-Meyerowitz conjecture holds in this much more general setting. However, the main result of this paper is to answer this question in the negative.
	
	\begin{theorem}\label{main}
		Let $p,q$ be prime numbers such that $p<q<p^2$, and set $M=p^4q^2$. Then there exists a functional pd-tiling $(f,g)$ of $\ZZ_M$ such that both $f$ and $g$ are step functions (as in Definition \ref{stepf}) and the corresponding polynomials $F, G$ satisfy condition (T1), but neither of them satisfies condition (T2).
	\end{theorem}
	
	Theorem \ref{CMthm} implies that the functional pd-tilings of Theorem \ref{main} cannot originate from proper tilings of $\ZZ_M$, i.e. there do not exist sets $A, B\subset \ZZ_M$ such that $A\oplus B=\ZZ_M$, and $f=h_A$, $g=h_B$.  However, in principle, it could be possible that one of the functions originates from a set, i.e. there exists a set $A\subset \ZZ_M$ such that $f=h_A$, but no corresponding set $B$ exists for which $g=h_B$ (or vice versa, $g$ originates from a set, but $f$ does not).
	
	\medskip
	
	A natural follow-up question is whether the examples $f\ast g=\mathbf{1}_{\ZZ_M}$ can be used in the future to construct proper counterexamples $A\oplus B=\ZZ_{M_1}$ to the Coven-Meyerowitz conjecture for some other value $M_1$. For example, one could hope to combine the functional pd-tilings of Theorem \ref{main} for $M=p^4 q^2$ and $M'=r^4s^2$, and try to construct a proper tiling counterexample in the group $\ZZ_{M_1}=\ZZ_{M M'}$.
	
	\medskip
	
	Note also, that by extending $f$ and $g$ periodically in an appropriate way,
	one may obtain such examples of functional pd-tilings for any $M'$ divisible by $M$. We also remark here that the examples in Theorem \ref{main} are far from being unique, e.g. other examples exist for $M=(pqr)^2$ (see Section \ref{app} for a detailed account of the case where $p=3, q=5, r=7$). On the other hand, it is not hard to prove that there are no examples in the spirit of Theorem \ref{main} for $M=pq$, and it also appears (by computer experiment) that there are no such examples for $M=p^\alpha q^\beta$ where $\beta=1,$ $\alpha\le 6$, or $\beta=2, \alpha\le 3$.
	
	\medskip
	
	The paper is organized as follows. In Section \ref{sec2} we explain the theoretical background of the linear programming formulation of the problem. This formulation gives a very convenient and efficient way to tackle the problem of functional tilings, while also providing non-trivial conditions for proper tilings. In Section \ref{sec3} we introduce some further auxiliary tools, similar to those developed in \cite{LaLo1}. Finally, in Section \ref{sec4} we  present a detailed account of the functional pd-tilings violating the (T2) condition for $M=p^4q^2$, thus settling Problem \ref{prob3}. We also give a minor positive result concerning Problem \ref{prob1}.
	
	\section{Linear programming constraints}\label{sec2}
	
	The aim of this section is to study tilings and functional pd-tilings with tools from linear programming. In particular, we will see that  questions of tiling are inherently related to finding maximal cliques in Cayley graphs, and Delsarte's linear programming bound is a natural tool to give upper bounds for the size of such cliques. As such, we can expect non-trivial tiling conditions to arise from linear programming methods.
	
	\medskip
	
	Given a set $A\subset \ZZ_M$, recall \cite{LaLo1} that the {\it divisor set} of $A$ is defined as $Div(A)=\{(a-a', M) : \ a, a'\in A\}$. We will use the notation $Div^\ast (A)=\cup_{d\in Div(A)}R_{d}$, i.e. $Div^\ast(A)$ is the union of step classes corresponding to the divisors of $M$ appearing in $Div(A)$. We will refer to such step classes $R_d$ ($d\in Div(A)$) as {\it divisor classes} corresponding to $A$. The following theorem of Sands is fundamental in the study of tilings of cyclic groups.
	
	\begin{theorem}[\cite{Sands}]\label{sands}
		For sets $A, B\subset \ZZ_M$, $A\oplus B=\ZZ_M$ is a tiling if and only if $|A||B|=M$ and $Div^\ast (A)\cap Div^\ast(B)=\{0\}$.
	\end{theorem}
	
	It is easy to see that this theorem translates questions of tiling into questions about maximal cliques in Cayley graphs defined on $\ZZ_M$. Namely, for a set $A\subset \ZZ_M$, define the Cayley graph $\Gamma_A$ as follows: the vertices of $\Gamma_A$ are the elements of $\ZZ_M$, i.e. $\{0, 1, \dots, M-1\}$, and two distinct vertices $x, y$ are connected if and only if $x-y\in Div^\ast (A)$. By Theorem \ref{sands} we conclude that $A$ tiles $\ZZ_M$ if and only if $|A|$ divides $M$ and there exists an independent set $B$ in the graph $\Gamma_A$ (i.e. a clique $B$ in the complement graph) of size $M/|A|$.
	
	\medskip
	
	As the graph $\Gamma_A$ depends only on $Div^\ast (A)$, it is natural to introduce the following general definition.

	\begin{definition}
		Given $M$, and a collection of step classes $H=\cup_{i=1}^r R_{m_i}$ with $0\in H$, we use the notation $H'=(\ZZ_M \setminus H) \cup \{0\}$ to denote the {\em standard complement} of $H$ (the word "standard" refers to the fact that $0$ appears in both $H$ and $H'$).
		We define the graph $\Gamma_H$ such that the vertices are the elements of $\ZZ_M$, i.e. $\{0, 1, \dots, M-1\}$, and two distinct vertices $x, y$ are connected if and only if $x-y\in H$. Also, we let $\omega(H)$ denote the maximal cardinality of a clique in the graph $\Gamma_H$.
	\end{definition}
	
	\medskip
	
	With this terminology at hand, we can formulate necessary and sufficient conditions for tilings with reference only to the sets $H$ and $H'$. It is easy to see that $\omega(H)\omega(H')\le M$ for any $H$ and, as a direct corollary of Theorem \ref{sands}, the case of equality is equivalent to tiling.
	
	\begin{proposition}\label{cl3}
		Let $H=\cup_{i=1}^r R_{m_i}$, $0\in H$ be given. There exists a tiling $A\oplus B=\ZZ_M$ with $Div^\ast (A)\subset H$ and $Div^\ast (B)\subset H'$ if and only if $\omega(H) \omega(H')=M$.
	\end{proposition}
	
	This means that given $M$ and $H$ we can, in principle, decide whether there exists a tiling $A\oplus B=\ZZ_M$ with $Div^\ast (A)\subset H$ and $Div^\ast (B)\subset H'$. In practice, however, determining the size of maximal cliques is an NP-hard problem. Of course, an upper bound on the clique number can be given by linear programming (there are several formulations of this bound -- we prefer to call it the {\it Delsarte LP bound}, cf \cite{MR}). We recall the following terminology from \cite{MR}.
	
	\medskip
	
	\begin{definition}\label{dp}
		For a collection of step classes $H=\cup_{i=1}^r R_{m_i}$, $0\in H$, and any positive number $\delta>0$, the Delsarte parameters corresponding to $H$ and $\delta$ are defined as follows:
		
		\noindent
		$D^+(H)=\max \{\sum_{z\in \ZZ_M} h(z): h\in\mathcal{A}(M), h(0)=1, \hat{h}\ge 0, h|_H\ge 0, h|_{H^c}=0\}$,
		
		\noindent
		$D^-(H)=\max \{\sum_{z\in \ZZ_M} h(z): \ h\in\mathcal{A}(M), h(0)=1, \hat{h}\ge 0, h|_{H^c}\le 0\}$,
		
		\noindent
		$D^{\delta+}(H)=\max \{\sum_{z\in \ZZ_M} h(z): \ h\in\mathcal{A}(M), h(0)=1, \hat{h}\ge 0, h|_H\ge \delta, h|_{H^c}=0\}$,
		
		\noindent and their reciprocals,
		
		\noindent
		$\lambda^+(H)=1/D^+(H)$, $\lambda^-(H)=1/D^-(H)$, $\lambda^{\delta +}(H)=1/D^{\delta+}(H)$.
	\end{definition}
	
	\medskip
	
	We remark here that in \cite{MR} the parameters $\lambda^+, \lambda^-$ were defined, but it is slightly more convenient for us here to work with their reciprocals, $D^+$ and $D^-$.
	
	\medskip

In order to gain some intuition of the concepts $D^+$, $D^-$ and $D^{\delta+}$, it is easiest to think of the continuous analogue of these quantities. If $B_1$ is the unit ball in $\RR^d$ (take $d=8$ for concreteness), then $D^+(B_1)$ is  the maximal integral of a nonnegative, positive definite function supported on $B_1$. It is known \cite{kr}, that the unique extremizer in this case is $\frac{1}{{\text{Vol}}(B_{1/2})}1_{B_{1/2}}\ast 1_{B_{1/2}}$, and $D^+(B_1)={\text{Vol}}(B_{1/2})$. For $D^{\delta+}(B_1)$  we insist that the function is at least $\delta$ over the entire ball. As such, $D^{\delta+}(B_1)$ is strictly less than $D^+(B_1)$, because the unique extremizer for $D^+(B_1)$ above is zero on the boundary. On the contrary, for $D^-(B_1)$ we allow the function to take any non-positive values outside the ball, which results in greater freedom. The quantity $D^-(B_1)$ gives a strong upper bound the density of sphere packings (in fact, the exact bound in dimension 8, c.f. \cite{via}), which is better than the volume bound given by $D^+(B_1)$. As such, for the ball $B_1$ in $\RR^8$ we have $D^{\delta+}(B_1)<D^+(B_1)<D^-(B_1)$ with strict inequalities at both places. Unfortunately, the situation is not so clear in the discrete setting of $\ZZ_M$. It is an interesting project to give necessary (or sufficient) conditions for the quantities  $D^{\delta+}(H), D^+(H), D^-(H)$ to coincide or differ.

\medskip

	There is an obvious hierarchy among the numbers $D^{\delta+}(H), D^+(H), D^-(H)$, and two of them give an upper bound on the clique number of $H$.
	
	\begin{proposition}\label{dmon}
		For any $H$ and any $\delta>0$, we have $D^{\delta +}(H)\le D^+(H)\le D^-(H)$, and
		$\omega(H)\le D^+(H)$.
	\end{proposition}
	
	\begin{proof}
		The first inequalities are obvious because of the hierarchy of the constraints in the linear programs defining them. The second inequality follows from the fact that for any clique $A$ in $\Gamma_H$ the function $h_A$ defined in \eqref{avg} is an admissible function in the definition of $D^+(H)$, and $\sum_{z\in \ZZ_M} h_A(z)=|A|$. (The function $h_A$ is indeed admissible, as it is a nonnegative, positive definite step function, supported on $\cup_{r\in R_1} r(A-A)\subset H$, and   $h_A(0)=1$.)
	\end{proof}
	
	\medskip
	
	There is an important duality between these Delsarte parameters.
	
	\begin{proposition}[\cite{MR}]\label{dual}
		For any $H$ we have $D^+(H)D^-(H')=M$.
	\end{proposition}
	\begin{proof}
		This is the content of Theorem 4.2 in \cite{MR}.
	\end{proof}
	
	We remark here that the main use of the Delsarte parameters is to give an upper bound on the clique numbers of Cayley graphs via this duality. Indeed,
	$\omega(H)\le D^+(H)=\frac{M}{D^-(H')}\le \frac{M}{\hat{h}(0)}$ for any admissible function $h$ in the definition of $D^-(H')$. Therefore, any admissible function leads to an upper bound on $\omega(H)$.
	
	\medskip
	
	Next, we give a connection of these parameters to tilings.
	
	\begin{proposition}\label{divh}
		Let $H=\cup_{i=1}^r R_{m_i}$, $0\in H$ be given. If there exists a tiling $A\oplus B=\ZZ_M$ with $Div^\ast (A)\subset H$ and $Div^\ast (B)\subset H'$, then  $|A|=\omega(H)=D^+(H)=D^-(H)$ and $|B|=\omega(H')=D^+(H')=D^-(H')$.
	\end{proposition}
	
	\begin{proof}
		Due to the assumed tiling, and Propositions \ref{dmon} and \ref{dual}, we have $|A|\le \omega(H)\le D^+(H)\le D^-(H)=\frac{M}{D^+(H')}\le \frac{M}{\omega(H')}\le \frac{M}{|B|}=|A|$. Therefore, equality must hold everywhere.
		
		The same argument holds for $H'$.
	\end{proof}
	
	\medskip
	
	There are two important points here: first, $D^+(H)=D^-(H)$ is a non-trivial linear programming constraint on $H$ and, second, the equalities $|A|=D^+(H)$, $|B|=D^+(H')$ determine the only possible sizes of $A$ and $B$ for a given $H$.
	
	\medskip
	
	Notice that we have assumed $Div^\ast (A)\subset H$. If, instead, we require $Div^\ast (A)=H$, we obtain the following additional  condition on $H$.
	
	\begin{proposition}\label{divh2}
		Let $H=\cup_{i=1}^r R_{m_i}$, $0\in H$ be given, and let $\delta=\frac{1}{M\phi(M)}$. If there exists a tiling $A\oplus B=\ZZ_M$ with $Div^\ast (A)= H$, then  $|A|=\omega(H)=D^{\delta+}(H)=D^+(H)=D^-(H)$, and $|B|=\omega(H')=D^+(H')=D^-(H')$.
	\end{proposition}
	
	\begin{proof}
		Due to the previous proposition, we only need to prove $|A|=D^{\delta+}(H)$.
		
		\medskip
		
		As $Div^\ast (A)=H$, the function $h_A$ defined in \eqref{avg} is an admissible function in the class of $D^{\delta+}(H)$. As such, $|A|\le D^{\delta+}(H)\le D^+(H)=|A|$, where the last equality follows from the previous proposition.
	\end{proof}
	So far, we have not considered the Coven-Meyerowitz conditions (T1) and (T2), due to the fact that these conditions are associated with the cyclotomic factors of $A(X)$ and $B(X)$, and not the divisor classes $Div^\ast(A), Div^\ast(B)$. We will now consider the Fourier transforms $\hat{h}_A$ and $\hat{h}_B$, and make use of condition (T1).
	
	\medskip

	Recall that for a divisor $d|M$ we have $\Phi_d(X)| A(X)$ if and only if $\hat{\mathbf{1}}_A(M/d)=0$, which happens if and only if $\hat{\mathbf{1}}_A$ vanishes on the class $R_{M/d}$.
	
	\medskip
	
	Assume that $A\oplus B=\ZZ_M$ is a tiling. As explained in the introduction, for any $d|M$ ($d\ne 1$) we have $\Phi_d(X)|A(X)$ or $\Phi_d(X)|B(X)$, or possibly both.
	However, when $d$ is a prime-power,  condition (T1) ensures that exactly one of the two divisibilities holds. That is, for any $p^\alpha|M$ we have either
	$\Phi_{p^\alpha}(X)|A(X)$ or $\Phi_{p^\alpha}(X)|B(X)$, but not both. Also, condition (T1) makes a connection between the cardinalities of $A$ and $B$, and the cyclotomic divisors $\Phi_{p^\alpha}(X)$. This motivates the following definition.
	
	\medskip
	
	\begin{definition}
		Let $H=\cup_{i=1}^r R_{m_i}$, $0\in H$ be given. Let $S_{H}$ denote the set of prime-powers $p^\alpha$ such that $M/p^\alpha$ is among the indices $m_i$. Let $k_H=\prod_{p^\alpha\in S_H} \Phi_{p^\alpha}(1)$.
	\end{definition}
	
	For any $H$ and its standard complement $H'$, and any prime power $p^\alpha|M$, the class $R_{M/p^\alpha}$ appears either in $H$ or in $H'$ (but not both). Also, $\Phi_{p^\alpha}(1)=p$. Therefore we have
	
	\begin{equation}\label{kh}
		k_H k_{H'}=M.
	\end{equation}

	\medskip
	
	The question arises: how can we recognize if $H=\cup_{i=1}^r R_{m_i}$, $0\in H$, is the support of the Fourier transform of $\mathbf{1}_A$ for a set $A$ appearing in a tiling $A\oplus B=\ZZ_M$. We have the following important necessary conditions.
	
	\begin{proposition}\label{hat1}
		Let $\delta=\frac{1}{M^2\phi(M)}>0$, and let $H=\cup_{i=1}^r R_{m_i}$, $0\in H$ be given. If there exists a tiling $A\oplus B=\ZZ_M$ with $H={\emph{supp}}|\hat{\mathbf{1}}_A|^2$, then $D^{\delta+}(H)=D^+(H)=D^-(H)=k_{H}=|B|$, and  $D^+(H')=D^-(H')=k_{H'}=|A|$.
	\end{proposition}

	\begin{proof}
		The proof exploits the fact that given any tiling $A\oplus B=\ZZ_M$, the functions $h_A$, $h_B$ (as defined in \eqref{avg}) and $\hat{h}_A$, $\hat{h}_B$ are all extremal in their respective $D^+$-classes. We can make this claim rigorous as follows.
		
		\medskip
		
		It is enough to prove
		
		\begin{equation}\label{h}
			D^{\delta+}(H)=D^+(H)=D^-(H)=k_{H}=|B|,
		\end{equation}
		because the second chain of equalities will be implied by the following duality relations (cf. Proposition \ref{dual} and equation \eqref{kh}):
		
		$$
		D^+(H')=\frac{M}{D^{-}(H)}, D^-(H')=\frac{M}{D^{+}(H)}, k_{H'}=\frac{M}{k_H}, |A|=\frac{M}{|B|},
		$$
		and all quantities are equal to the right hand side by \eqref{h}, so the left hand sides must also be equal.
		
		\medskip
		
		To prove \eqref{h}, notice first the following equivalent statements for any prime power $p^\alpha|M$: \\
        $p^\alpha\in S_A \Leftrightarrow \Phi_{p^\alpha}(X)|A(X)\Leftrightarrow \hat{1}_A$ vanishes on $R_{M/p^\alpha}$
        $\Leftrightarrow R_{M/p^\alpha}\subset H' \Leftrightarrow p^\alpha\in S_{H'}$. The first equivalence
        is just the definition of $S_A$, the second follows form \eqref{cycfou2}, the third follows from the assumption that the support of $\hat{1}_A$ is exactly $H$, while the last is just the definition of $S_{H'}$. Hence, invoking the  (T1) tiling condition of Coven-Meyerowitz in Theorem \ref{CMthm}, we obtain $|A|=\prod_{p^\alpha\in S_A}\Phi_{p^\alpha}(1)=\prod_{p^\alpha\in S_{H'}}\Phi_{p^\alpha}(1)=k_{H'}$.
        Furthermore, by $A\oplus B=\ZZ_M$ we have $|A||B|=M$, and by equation \eqref{kh} we have $k_Hk_{H'}=M$, which implies $|B|=k_{H}$.
		
		\medskip
		
		Furthermore, let $h_A$ be as in equation \eqref{avg}, and $u=\frac{1}{|A|}\hat{h}_A$. Then $\supp u =H$ by assumption,
		$u(0)=1$, $\hat u\ge 0$, and due to \eqref{trh} we have $u(z)\ge \frac{1}{|A|^2\phi(M)}\ge\delta$ for all $z\in H$ . This shows that $u$ is an admissible function in the class $D^{\delta +}(H)$. As  $\sum_{z\in \ZZ_M} u(z)= \hat{u}(0)=\frac{M}{|A|}h_A(0)=|B|$, we conclude $|B|\le D^{\delta +}(H)$.
		
		\medskip
		
		Similarly, for $v=\frac{1}{|B|} \hat{h}_B$ we have that $\supp v\subset H'$ (by the tiling condition the supports must be essentially disjoint),  $v(0)=1$, $\hat v\ge 0$, $\sum_{z\in \ZZ_M} v(z)=|A|$. This shows that $|A|\le D^+(H')$.
		
		\medskip
		
		Finally, putting everything together, we have $k_H=|B|\le D^{\delta +}(H)\le D^{+}(H)\le D^{-}(H)=\frac{M}{D^+(H')}\le \frac{M}{|A|}=|B|$, and hence equality must hold everywhere.
	\end{proof}
	
	The chain of equalities $D^{\delta+}(H)=D^+(H)=D^-(H)=k_H$ puts severe restrictions on $H$. While these restrictions fall short of proving the Coven-Meyerowitz conjecture, they rule out the vast majority of sets $H=\cup_{i=1}^r R_{m_i}$ as candidates for the support of the Fourier transform of a tile in $\ZZ_M$. As an indication, in Section \ref{app} we will provide some statistics of a computer experiment carried out for a specific value of $M$.
	
	\section{Cyclotomic divisibility for averaged functions}\label{sec3}
	
	In connection to cyclotomic divisibility, the class of step functions enjoys another useful property. Strictly speaking, the results of this section are not needed to verify the forthcoming examples in Section \ref{sec4}. However, we include them for two reasons. One is that Proposition \ref{avgcyc} below gives a potentially useful tool in studying cyclotomic divisibility of step functions in the future. The second is that we have actually used these conditions in {\it producing} the examples presented in Section \ref{sec4}.

    \medskip

    In service of our discussion we borrow some notation  from \cite{LaLo1} for (classic) $M$-cuboids.
	\begin{definition}
		Let $M=\prod_{i=1}^K p_i^{n_i}$, where $p_i$ are distinct primes and $n_i\geq 1$ are integers. Let $f:\ZZ_M\rightarrow \RR$.
		\begin{itemize}
			\item An $M$-cuboid is a weighted set corresponding to the mask polynomial of the form
			
			\begin{equation}\label{cub}\Delta(X)=X^{c} \prod_{i=1}^K (1-X^{d_i}),
			\end{equation}
			where $c\in \ZZ_M, d_i=\rho_i M/p_i$ with $\rho_i\in \{1,\ldots, p_i-1\}$ for all $i$.
		
			\item The associated $\Delta$-evaluation of $f$ is
			
			\begin{equation}\label{cubeval}
				F[\Delta]=\sum_{\overrightarrow{\epsilon}\in \{0,1\}^K} (-1)^{\sum_{j=1}^K \epsilon_j} f(x_\epsilon),
			\end{equation}
			where $x_\epsilon=c+\sum_{j=1}^K \epsilon_jd_j$ {\em (mod} $M${\em )} are the vertices of $\Delta$.
		\end{itemize}
	\end{definition}
	
	There is a well-known cyclotomic divisibility test given in terms of cuboids, which was first observed by Steinberger \cite{Steinberger}, and later reestablished in \cite{KMSV, LaLo1}.

	\begin{proposition}\label{cuboid}
		Let $f:\ZZ_M\rightarrow\RR$. Then the following are equivalent:
		\begin{enumerate}
			\item[(i)] $\Phi_M(X)|F(X)$,
			\item[(ii)] For all $M$-cuboids $\Delta$, we have $F[\Delta]=0$.
		\end{enumerate}
	\end{proposition}
	
	Given a function $f:\ZZ_M\rightarrow\RR$, cyclotomic divisibility of $f$ by $\Phi_N$ with $N|M$ is determined by Proposition \ref{cuboid} with $M$ and $f$ replaced by $N$ and $f_N$ respectively, where $f_N$ is the function defined on $\ZZ_N$
	\begin{equation}\label{fmodN}
		f_N(z)=\sum_{N|z-y} f(y).
	\end{equation}

	We will show that when $f\in \cala(M)$ is a step  function, Proposition \ref{cuboid} considerably simplifies. We first need an auxiliary result.
	
	\begin{lemma}\label{adjecantver}
		Let $\Delta$ be an $M$-cuboid and $f\in \cala(M)$. Suppose that there exists $i\in\{1,\ldots, K\}$ and $\beta\leq n_i-1$ such that for all vertices $x_\epsilon$  of $\Delta$ we have
		\begin{equation}\label{ver_dist_all}
			(x_\epsilon,p_i^{n_i})=p_i^{\beta}.
		\end{equation}
		Then for any pair of vertices $x_\epsilon, x_{\epsilon'}$ of $\Delta$ with
		\begin{equation}\label{ver_dist_pair}
			(x_\epsilon- x_{\epsilon'},M)=M/p_i,
		\end{equation}
		we have
		\begin{equation}\label{sameclass}
			f(x_\epsilon)=f(x_{\epsilon'}).
		\end{equation}
		Moreover, if\eqref{ver_dist_all} holds with $\beta\leq n_i-2$ for one single vertex of $\Delta$, then it holds for all the vertices with the same exponent $\beta$, and hence \eqref{sameclass} holds for any pair of vertices of $\Delta$ satisfying \eqref{ver_dist_pair}.
	\end{lemma}
	\begin{proof}
		Let $m|M$, and suppose that  $x_\epsilon\in R_m$. As $f$ is a step function, it is enough to show that $x_{\epsilon'}\in R_m$.

        The exponent of $p_i$ in $m$ is $\beta\le n_i-1$ by \eqref{ver_dist_all}. If $\beta\le n_i-2$, then by \eqref{ver_dist_pair} we have $x_{\epsilon'}\equiv x_\epsilon -cp_1^{n_1}\dots p_i^{n_i-1}\dots p_K^{n_K}$ (mod $M$), with some $1\le c\le p_i-1$. Hence we obtain $x_{\epsilon'}\in R_m$, because the exponents of all the primes in $x_{\epsilon'}$ remain the same as in $x_\epsilon$.
		Assume now that $\beta={n_i-1}$. In $x_{\epsilon'}$ the exponents of primes different from $p_i$ still remain the same as in $x_\epsilon$, and that of $p_i$ either remains $n_i-1$ or becomes $n_i$. However, the latter is excluded by assumption \eqref{ver_dist_all}. Again, we conclude $x_{\epsilon'}\in R_m$.
		
		\medskip
		
		For the last sentence of the proposition,  notice that the difference between any two vertices of an $M$-cuboid is divisible by $p_1^{n_1-1}\dots  p_K^{n_K-1}$. Therefore, if the exponent of $p_i$ in $x_\epsilon$ is $\beta\le n_i-2$, then it is the same for all the other the vertices.
		
	\end{proof}
	
	\begin{corollary}\label{mnull}
		Let $\Delta$ be an $M$-cuboid and $f\in \cala(M)$. Suppose that there exists $i\in\{1,\ldots, K\}$, and $\beta\le n_i-1$ such that \eqref{ver_dist_all} holds for all vertices $x_\epsilon$  of $\Delta$. Then $F[\Delta]=0$.
	\end{corollary}
	\begin{proof}
		Write $\Delta$ as a disjoint union of its two facets $\Delta_i\cup\Delta'_i$, with respect to $p_i$: explicitly, each facet is defined as the maximal subset of vertices such that the pairwise difference between any two elements is divisible by $p_i^{n_i}$. The vertices of the facets $\Delta_i$ and $\Delta'_i$ appear with opposite signs in the evaluation \eqref{cubeval}. Also, for every element $v$ of $\Delta_i$ there exists a unique element $v'$ of $\Delta'_i$ such that $(v-v',M)=M/p_i$. By Lemma \ref{adjecantver} $f(v)=f(v')$ which implies $F[\Delta]=0$.
	\end{proof}
	
	We are now ready to prove the strengthening of Proposition \ref{cuboid} for step functions.
	\begin{proposition}\label{avgcyc}
		Let $f\in \cala(M)$. The following are equivalent:
		\begin{enumerate}
			\item[(I)] $\Phi_M(X)|F(X)$,
			\item[(II)] $F[\Delta]=0$ for all $M$-cuboids $\Delta$ with a vertex at $0$,
			\item[(III)] $F[\Delta]=0$ for one $M$-cuboid $\Delta$ with a vertex at $0$.
		\end{enumerate}
	\end{proposition}
	\begin{proof}
		By Proposition \ref{cuboid} (I) implies (II). Clearly (II) implies (III).
		
		We prove (II) implies (I). By Proposition \ref{cuboid} we need to prove $F[\Delta]=0$ for all $M$-cuboids $\Delta$. By Corollary \ref{mnull} it is enough to consider $M$-cuboids which violate \eqref{ver_dist_all} for all $i$. Fix an index $i$. By the last part of Lemma \ref{adjecantver}, the condition \eqref{ver_dist_all} can only be violated if at some vertices of the cuboid the exponent of $p_i$ is equal to $n_i-1$, while at other vertices it is equal to $n_i$.

		That is, for every $i\in \{1,\ldots, K\}$ there exists a vertex $x_i$ of $\Delta$ satisfying
		\begin{equation}\label{planerest}
			p_i^{n_i}|x_i.
		\end{equation}
		Now, for every index $i$, partition $\Delta$ into two facets, with respect to $p_i$, $\Delta=\Delta_i\cup \Delta'_i$ such that for every $i=1,\ldots, K$ $x_i\in \Delta_i$. By definition $p_i^{n_i}|x-x_i$ for all vertices $x$ of $\Delta_i$, and hence  by \eqref{planerest} we have
		$$
		p_i^{n_i}|x.
		$$
		Let $v=\cap_{i=1}^K \Delta_i$ (the intersection is non-empty, as $K$ facets of different orientation of $\Delta$ always intersect at exactly one point -- just like $K$ axis-parallel hyperplanes intersect in one point in $\RR^K$). Clearly, $p_i^{n_i}|v$ for all $i$, and this  is equivalent to $v=0$.	By (II) we have $F[\Delta]=0$.
			
		In order to prove (III) implies (II) we take two $M$-cuboids $\Delta, \Delta'$ so that $0$ is a joint vertex of both. Let $v=\sum_{j=1}^K\epsilon_jd_j$ be a vertex of $\Delta$ and $v'=\sum_{j=1}^K\epsilon_jd_j'$ the corresponding vertex of $\Delta'$, with the same coefficients $\epsilon_j$. As such, the vertices $v$ and $v'$ have the same sign in the evaluation \eqref{cubeval} corresponding to $\Delta$ and $\Delta'$, respectively. Also, it is clear that $(v,M)=(v',M)$, and hence $v$ and $v'$ fall into the same step class, implying $f(v)=f(v')$.  Therefore, the terms in the evaluation \eqref{cubeval} for $\Delta$ are exactly the same as those for $\Delta'$, and hence  we have $F[\Delta]=0$ if and only if $F[\Delta']=0$.
		\end{proof}
		
		In practical terms, when $M=p^nq^m$, $p$ and $q$ are primes, and $m,n\in \NN$, Proposition \ref{avgcyc} implies that for any $f\in \cala(M)$ and any $N|M$ divisible by $pq$
		
		\begin{equation}\label{fn}
			\Phi_N(X)|F(X) \text{ if and only if } f_N(0)-f_N(N/p)-f_N(N/q)+f_N(N/pq)=0,
		\end{equation}
		where $f_N$ is defined in \eqref{fmodN}. When $N$ is coprime to $p$ or $q$, i.e. a prime power, the relation is as follows. Let $N=p^\alpha, \alpha\geq 1$, then
		
		\begin{equation}\label{pri}
			\Phi_N(X)|F(X) \text{ if and only if } f_N(0)=f_N(p^{\alpha-1}).
		\end{equation}

		\section{Functional pd-tilings violating the (T2) condition}\label{sec4}
		
		In this final section we first present a minor positive result concerning Problem \ref{prob1}, and then provide a family of functional pd-tilings which violate the (T2) condition of Coven-Meyerowitz for $M=p^4q^2$.
		
		\medskip
		
		\begin{proposition}\label{primep}
			Let $M$ be a prime-power, $M=p^\alpha$. In this case, any set $A$ which pd-tiles $\ZZ_M$, also tiles $\ZZ_M$ properly.
		\end{proposition}
		
		\begin{proof}
			
		As the Coven-Meyerowitz conjecture is true for $M=p^\alpha$, it is enough to prove conditions (T1) and (T2) for $A$, and as (T2) is vacuous in this case, it is sufficient to  prove (T1) only.
			
		\medskip
			
		Let $\mathbf{1}_A\ast f=\mathbf{1}_{\ZZ_M}$ be a pd-tiling. By \eqref{pdt} the support of $\hat{\mathbf{1}}_A$ and $\hat{f}$ intersect only at zero. Let $H=\{0\} \cup_{k =1}^r R_{p^{j_k}}$ denote the support of  $\hat{\mathbf{1}}_A$ (the support is necessarily a union of step classes, as noted in the Introduction). As  $r$ denotes the number of step classes in the support of $\hat{1}_A$, there are $\alpha-r$ cyclotomic divisors of $A(X)$, and hence condition (T1) means that $|A|=p^{\alpha-r}$.
			
		\medskip
		
		Consider the "standard" tiling $C\oplus D=\ZZ_M$, corresponding to $H$ and its complement $H'$, where the mask polynomial of $C$ is $\prod_{k=1}^r \Phi_{M/p^{j_k}}(X)$, and the mask polynomial of $D$ is the product of the remaining cyclotomic polynomials $\Phi_{M/p^s}$ ($s\ne j_k$). In this tiling $|C|=p^r$,  $|D|=p^{\alpha-r}$, and the support of $\hat{\mathbf{1}}_C$ is $H'$, and the support of $\hat{\mathbf{1}}_D$ is $H$. Hence, Proposition \ref{hat1} shows that $D^+(H)=p^r$, $D^+(H')=p^{\alpha-r}$.
			
		\medskip
			
		On the other hand, we claim that $D^+(H')=|A|$. Informally, this is true because the functions $\frac{1}{|A|^2}|\hat{\mathbf{1}}_A|^2$ and $\frac{1}{\hat{f}(0)}\hat{f}$ are both extremal in the classes defining $D^-(H)$ and $D^+(H')$, respectively. Formally, the functions clearly belong to these classes, and $\frac{1}{|A|^2}\sum_{\xi\in \ZZ_M} |\hat{\mathbf{1}}_A|^2(\xi)=\frac{M}{|A|}$, $\frac{1}{\hat{f}(0)}\sum_{\xi\in\ZZ_M}\hat{f}(\xi)=\frac{M}{|F|}=|A|$, and hence (invoking Porposition \ref{dual} in the middle equality)
		$$
			|A|=\frac{M}{\frac{1}{|A|^2}\sum_{\xi\in \ZZ_M} |\hat{\mathbf{1}}_A|^2(\xi)}\ge \frac{M}{D^-(H)}=D^+(H')\ge \frac{\sum_\xi \hat{f}(\xi)}{\hat{f}(0)}=|A|,
		$$
		and equality must hold everywhere.
			
			\medskip
		
		In summary, we see that $p^{\alpha-r}=D^+(H')=|A|$ which proves exactly (T1).
		\end{proof}
		
		In the terminology of \cite{KMMS} this means that the group $\ZZ_M$ is {\it pd-flat} for $M=p^\alpha$.
		We do not know whether the analogous statement holds true for $M=p^\alpha q^\beta$. If so, it would imply the "spectral $\to $ tile" direction of Fuglede's conjecture in those groups, which would be a highly nontrivial result (for the sharpest currently known results in those groups see \cite{M}).
		
		\medskip
		
		We now present the proof of Theorem \ref{main}, by describing examples of functional pd-tilings that violate the Coven-Meyerowitz (T2) condition for $M=p^4q^2$ under the additional assumption
		\begin{equation}\label{order}
		p<q<p^2.
		\end{equation}

		\begin{proof}[Proof of Theorem \ref{main}]
			
		The functions $f$ and $g$ we present below both belong to the class $\cala(M)$, their supports intersect only at the origin, and they are both eigenfunctions of the Fourier transform operator in $\ZZ_M$ with eigenvalue $\sqrt{M}=p^2q$. This is sufficient in order to prove that they tile $\ZZ_M$ since \eqref{fweaktile} is equivalent to
		$$
		\hat{f}\cdot\hat{g}=M \delta_0,
		$$
		which is clearly satisfied in this case.

		Let
		$$
		f=\sum_{m\in \Div(f)} c_m \mathbf{1}_{R_m},\ \ \ \, g=\sum_{m\in \Div(g)} \nu_m \mathbf{1}_{R_m},
		$$
		where we slightly abuse notation, setting $\Div(f)=\{m|M\,:\, f(m)>0\}$ and similarly for $g$. We display the functions in 2-dimensional arrays in which the $p$-axis is indexed by $M/p^\alpha$ ($0\le \alpha\le 4$), and the $q$-axis is indexed by $M/q^\beta$ ($0\le\beta\le 2$). The array entry $(M/p^\alpha, M/q^\beta)$ will correspond to the value of the function on the step class $R_{M/p^\alpha q^\beta}$, i.e. either $c_{M/p^\alpha q^\beta}$ or $\nu_{M/p^\alpha q^\beta}$.
			
		Consider the function $f$ defined as follows
		\begin{center}
		\begin{tabular}{ c|| c | c | c }
			$f$&$M$ & $M/q$ & $M/q^2$ \\ [0.5ex]
			\hline\hline
		
			$M$ & $1$ & 0 & $(q-p)/\phi(q^2)$\\[0.5ex]
			\hline
					
			$M/p$ & 0 & $1/\phi(q)$ & $(q-p)/\phi(q^2)$\\[0.5ex]
			\hline
					
			$M/p^2$ &$(q^2-pq+p^2-q)/p\phi(q^2)$ &$(p^2-q)/p\phi(q^2)$ &0 \\[0.5ex]
			\hline
		
			$M/p^3$ & $(q-p)/p\phi(q)$ & 0 & 0\\[0.5ex]
			\hline
		
			$M/p^4$ & 0 & 0 & $1/p^2\phi(q)$ \\[0.5ex]
			\hline
		\end{tabular}
		\end{center}
			
		Set $d=2pq-p^2-q$. The function $g$ is as follows
			
		\begin{center}
		\begin{tabular}{ c|| c | c | c }
			$g$ &$M$ & $M/q$ & $M/q^2$ \\ [0.5ex]
			\hline\hline
			
			$M$ & $1$ & $p(q-p)/d$ & 0\\[0.5ex]
			\hline
		
			$M/p$ &  $q\phi(p)/d$  & 0 & 0\\[0.5ex]
			\hline
			
			$M/p^2$ & 0 & 0 & $1/d$ \\[0.5ex]
			\hline
			
			$M/p^3$ & 0 & $q/pd$ & $\phi(p)/pd$\\[0.5ex]
			\hline
			
			$M/p^4$ & $(q-p)/pd$ & $(q-p)/pd$ & 0 \\[0.5ex]
			\hline
		\end{tabular}
		\end{center}
		
		By the restriction \eqref{order} it is easy to verify  that both $f$ and $g$ are nonnegative.
			
		\medskip
			
		To calculate the Fourier transforms of $f$ and $g$ one can use the  explicit formula given in Lemma \ref{fl} (because $f$ and $g$ are linear combinations of indicator functions of step classes). Using that formula one can directly verify that $f$ and $g$ are indeed eigenfunctions of the Fourier transform with eigenvalue $\sqrt{M}=p^2q$, and consequently all conditions in Definition \ref{fweakdef} are satisfied.
			
		\medskip
			
		Considering the zero sets of $f$ and $g$, while keeping in mind that $c_m=0$ if and only if $\Phi_{M/m}(X)|F(X)$ (and similarly for $\nu_m$ and $g$), we see that
		$$
		\Phi_p\Phi_{p^4}\Phi_q|F(X),\,\Phi_{p^2}\Phi_{p^3}\Phi_{q^2}|G(X).
		$$
		On the one hand, a direct calculation shows $\hat{f}(0)=\hat{g}(0)=p^2q$, hence condition (T1) is satisfied. On the other hand, we observe that $c_{M/pq}$ and $ \nu_{M/p^2q^2}$ are positive therefore, respectively, $\Phi_{pq}\nmid F(X)$ and $\Phi_{p^2q^2}\nmid G(X)$ which in turn implies that both functions violate  condition (T2).
		\end{proof}
		
		\medskip
		
		We note that cyclotomic divisibility of $f$ and $g$ can also be explicitly verified through Proposition \ref{avgcyc}, which provides a system of linear constraints on the coefficients $c_m,\nu_m$. For example, by Proposition \ref{avgcyc} the fact that $\Phi_M(X)|G(X)$ is equivalent to
		$$
		\nu_M-\nu_{M/p}-\nu_{M/q}+\nu_{M/pq}=0,
		$$
		which evidently holds true.	Similarly, the fact that
		$\Phi_{M/p}|F(X)$, is equivalent to $\Phi_{M/p}|F_{M/p}(X)$ (where the polynomial $F_{M/p}(X)$ is defined by \eqref{maskpoly} for the function $f_{M/p}$  in \eqref{fmodN}), and by \eqref{fn} this can be further restated as
		
		$$
		c_M-\phi(p)c_{M/pq}-pc_{M/p^2}+pc_{M/p^2q}=0,
		$$
		and directly verified. All cyclotomic divisibility conditions can be verified in this manner (as an alternative way, instead of checking that $f$ and $g$ are eigenfunctions of the Fourier transform).
		
\appendix 
\section*{Appendix}
\label{app}
		
		Finally, in order to demonstrate the strength of Proposition \ref{hat1} we include here some statistical data of a computer experiment for a specific (fairly small) value of $M$.
		
		\medskip
		
		Let $M=3^25^27^2$, and let us try to make a list of those subsets $H\subset \ZZ_M$ which satisfy the conditions of Proposition \ref{hat1} with $k_H=pqr$.
		
		\medskip
		
		There are 27 divisors $m$ of $M$, and correspondingly 27 classes $R_m$. As a first step, from each pair of classes $\{R_{M/p}, R_{M/p^2}\}$, $\{R_{M/q}, R_{M/q^2}\}$, $\{R_{M/r}, R_{M/r^2}\}$ we put one into $H$, and the other into the standard complement $H'$. We refer to these as the 'prime-power classes'. There are $2^3=8$ ways to make this selection. We put $0$ into both $H$ and $H'$ in all cases.
		As such, after distributing the prime-power classes, 7 of the classes $R_m$ are positioned into $H$ or $H'$. There are 20 remaining classes $R_{M/d}$ (corresponding to the non-prime-power divisors $d$). There are $2^{20}$ ways to select which ones we put into $H$. For each selection we run two  linear programs (LPs) to test whether $D^{\delta+}(H)=D^-(H)$, and whether $D^{\delta+}=pqr$, i.e whether $H$ satisfies the conditions of Proposition \ref{hat1}. This means running $2^3\cdot 2^{20}\cdot  2\approx 16\cdot 10^6$ LP's. Finally, if $H$ satisfies the conditions of Proposition \ref{hat1}, we can check whether it violates (T2) or not.
		
		\medskip
		
		The results are as follows. For each selection of prime-power classes in $H$, the total number of possible selections of non-prime-power classes is $2^{20}$, and we indicate the number of those which satisfy the conditions of Proposition \ref{hat1}, and which do not satisfy the (T2) condition:
		
		\medskip

		\begin{center}
		\begin{tabular}{ ||c c c c|| }
		\hline
		Prime-powers & Total cases & $H$ satisfying Prop. \ref{hat1} & $H$ violating (T2) \\ [0.5ex]
		\hline\hline
		$\{3, 5, 7\}$ & $2^{20}$ & $10796$ & $2$ \\[0.5ex]
		\hline
		$\{3, 5, 49\}$ &  $2^{20}$ & $5384$ & $5$\\[0.5ex]
		\hline
		$\{3, 25, 7\} $ & $ 2^{20}$ & $6164$ & $ 11$\\[0.5ex]
		\hline
		$\{3, 25, 49\}$ &  $2^{20}$ & $2190$ & $18$\\[0.5ex]
		\hline
		$\{9, 5, 7\}$ & $2^{20}$ & $5523$ & $50$\\[0.5ex]
		\hline
		$\{9, 5, 49\}$ &  $2^{20}$ & $2834$ & $12$\\[0.5ex]
		\hline
		$\{9, 25, 7\}$ & $2^{20}$ & $3190$ & $2$\\[0.5ex]
		\hline
		$\{9, 25, 49\}$ & $2^{20}$ & $1281$ & $13$\\ [0.5ex]
		\hline
		\end{tabular}
		\end{center}
		
		\medskip

		In summary, the total number of sets $H$ satisfying the conditions of Proposition \ref{hat1} is 37362, which may seem large, but recall that we had about 8 million sets $H$ to test, so these represent less than half percent of the cases. The vast majority of these satisfy condition (T2). The total number of sets $H$ which violate (T2) is only 113 out of 37362. For each such  $H$,  we can obtain a functional pd-tiling $f\ast g=\mathbf{1}_{\ZZ_M}$ such that $f$ violates the (T2) condition. Interestingly, in all such examples we have found, $g$ also violates the (T2) condition. In some of these examples $f$ and $g$ are eigenfunctions of the Fourier transform (as in Theorem \ref{main}), while in others they are not.
		
		\medskip
		
		It remains to be determined by future research  whether these functional counterexamples can be used as building blocks of a proper counterexample to the Coven-Meyerowitz conjecture.


\section*{Acknowledgments} 
The authors are very grateful to the reviewers for their insightful comments and suggestions, and for spotting misprints and inaccuracies in the original draft. Also, the lower bound \eqref{trh} was suggested by one of the reviewers (with a different argument leading to it).

\medskip

        The research was partly carried out at the Erd\H os Center, R\'enyi Institute, in the framework of the semester "Fourier analysis and additive problems".

		G.K. was supported by the Hungarian National Foundation for
		Scientific Research, Grants No. K146922, FK 142993 and by the J\'anos Bolyai Research Fellowship of the Hungarian Academy of Sciences.
		
		I.L. is supported by the Israel Science Foundation (Grant 607/21).
		
		M.M. was supported by the Hungarian National Foundation for Scientific Research, Grants No. K132097, K146387.
		
		G.S. was supported by the Hungarian National Foundation for Scientific Research, Grant No. 138596 and Starting Grant 150576.

\bibliographystyle{amsplain}


\begin{dajauthors}
\begin{authorinfo}[GK]
  Gergely Kiss\\
   Corvinus University of Budapest, Department of Mathematics \\
 Fővám tér 13-15, Budapest 1093, Hungary,\\
  and\\
  HUN-REN Alfr\'ed R\'enyi Institute of Mathematics\\
 Re\'altanoda utca 13-15, H-1053, Budapest, Hungary\\
 kiss.gergely\imageat{}renyi\imagedot{}hu \\

\end{authorinfo}
\begin{authorinfo}[IL]
  Itay Londner\\  Department of Mathematics, Faculty of Mathematics and Computer Science,\\ Weizmann Institute of Science,\\ Rehovot 7610001, Israel\\ 
itay.londner\imageat{}weizmann\imagedot{}ac\imagedot{}il\\
\end{authorinfo}
\begin{authorinfo}[MM]
        M\'at\'e  Matolcsi\\
		HUN-REN Alfr\'ed R\'enyi Institute of Mathematics\\
		Re\'altanoda utca 13-15, H-1053, Budapest, Hungary\\
		and\\
		Department of Analysis and Operations Research,
		Institute of Mathematics,\\
		Budapest University of Technology and Economics,\\
		M\H uegyetem rkp. 3., H-1111 Budapest, Hungary,\\
  matomate\imageat{}renyi\imagedot{}hu\\
\end{authorinfo}
\begin{authorinfo}[GS]
 G\'abor Somlai \\
Eötvös Loránd University, Faculty of Science, Institute of Mathematics, Department of Algebra and Number Theory \\
 P\'azm\'any P\'eter s\'et\'any 1/C, Budapest, Hungary, H-1117
\\ and \\
HUN-REN Alfr\'ed R\'enyi Institute of Mathematics\\
 Re\'altanoda utca 13-15, H-1053, Budapest, Hungary,
 \\
 gabor.somlai\imageat{}ttk\imagedot{}elte\imagedot{}hu
\end{authorinfo}
\end{dajauthors}

\end{document}